\documentclass[11pt]{amsart}
\usepackage{amsmath,amsfonts,amsthm,graphicx, amssymb}

\newtheorem{proposition}{Proposition}[section]
\newtheorem{theorem}[proposition]{Theorem}
\newtheorem{lemma}[proposition]{Lemma}
\newtheorem{corollary}[proposition]{Corollary}

\theoremstyle{definition}
\newtheorem{definition}[proposition]{Definition}

\theoremstyle{remark}
\newtheorem{remark}[proposition]{Remark}
\newtheorem{example}[proposition]{Example}

\def\Q{\mathbb{Q}}
\def\C{\mathcal{C}}
\def\R{\mathbb{R}}
\def\Z{\mathbb{Z}}
\def\hfs{S(S^1 \times D^2)}
\def\hfsone{S(S^1 \times S^2 ; R)}

\def\st{S^1\times \mathbb{R}^2}
\def\fig#1{\raisebox{-2.2ex}{\includegraphics[height=5.2ex]{#1}}}

\begin{document}

\title[On the $S^1\times S^2$ HOMFLY-PT invariant and Legendrian links]{On the $S^1\times S^2$ HOMFLY-PT invariant and Legendrian links}
\author{Mikhail Lavrov}
\address{Carnegie Mellon University, Pittsburgh, PA 15213}
\email{mlavrov@andrew.cmu.edu}

\author{Dan Rutherford}
\address{University of Arkansas, Fayetteville, AR 72701}
\email{drruther@uark.edu}

\begin{abstract}   
In \cite{GZ}, Gilmer and Zhong established the existence of an invariant for links in $S^1\times S^2$ which is a rational function in variables $a$ and $s$ and satisfies the HOMFLY-PT skein relations.  We give formulas for evaluating this invariant in terms of a standard, geometrically simple basis for the HOMFLY-PT skein module of the solid torus.  This allows computation of the invariant for arbitrary links in $S^1\times S^2$ and shows that the invariant is in fact a Laurent polynomial in $a$ and $z= s -s^{-1}$.  Our proof uses connections between HOMFLY-PT skein modules and invariants of Legendrian links.  As a corollary, we extend HOMFLY-PT polynomial estimates for the Thurston-Bennequin number to Legendrian links in  $S^1\times S^2$ with its tight contact structure.  
\end{abstract}

\maketitle

\section{Introduction}

The main results in this paper provide formulas for the computation of the HOMFLY-PT invariant of links in $S^1\times S^2$.  In order to give a more detailed overview it is convenient to first introduce some notations.

Let $M$ be a smooth, oriented $3$-manifold and denote by $\mathcal{L}(M)$ the set of isotopy classes of oriented, framed links in $M$ including the empty link.  If $M$ has boundary we only consider links that are disjoint from $\partial M$.

\begin{definition}  Let $R$ be a commutative ring
containing the ring $\Z[a^{\pm 1}, z^{\pm 1}]$.  The {\it HOMFLY-PT skein module of $M$}, $S(M ; R)$, is the quotient of the free $R$-module generated by $\mathcal{L}(M)$ by the submodule generated by the HOMFLY-PT skein relations:
\begin{equation*}
\tag{i}\fig{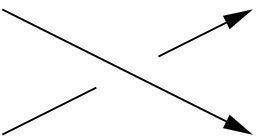} - \fig{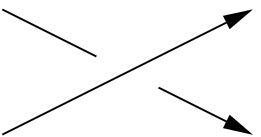} = z \fig{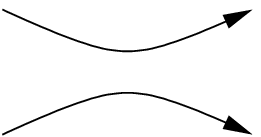},
\end{equation*}
\begin{equation*}
\tag{ii}\fig{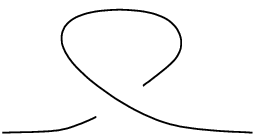} = a \raisebox{-.6ex}{\includegraphics[height=1.2ex]{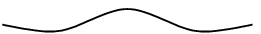}} \,\,\, \mbox{and} \,\,\, \fig{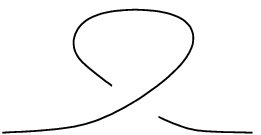} = a^{-1} \raisebox{-.6ex}{\includegraphics[height=1.2ex]{HSR7.eps}},
\end{equation*}
\begin{equation*}
\tag{iii}\displaystyle  L \sqcup \fig{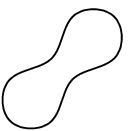} = \frac{a-a^{-1}}{z} L.
\end{equation*}
The links appearing in (i) and (ii) are assumed to agree outside of a ball where they differ as pictured.  In (iii), the unknotted component should sit in a ball that is disjoint from the rest of $L$. In all figures, links have the blackboard framing.    
\end{definition}

The skein module provides the universal generalization of the HOMFLY-PT invariant for links in $M$ in the following sense:  Any link invariant, $\mathcal{L}(M) \rightarrow N$ with values in an $R$-module, $N$, and satisfying the skein relations (i), (ii), and (iii) may be factored through the obvious map $\mathcal{L}(M) \rightarrow S(M;R)$.  Conversely, any element $g \in \mathit{Hom}_R( S(M;R), N)$ provides an $N$-valued HOMFLY-PT invariant for framed links in $M$.

Throughout this paper we will fix the coefficient ring $R \subset \Q(a,s)$ to be the smallest subring containing the Laurent polynomials $\Z[a^{\pm 1}, s^{\pm 1}]$ as well as the denominators $s^{k} - s^{-k}$, for $k \geq 1$.  Here, the inclusion of $\Z[a^{\pm 1}, z^{\pm 1}]$ into the field of rational functions in $a$ and $s$ arises from the identification, $z = s - s^{-1}$.  With this choice of $R$ fixed, we will omit the coefficient ring from our notation for skein modules.

Gilmer and Zhong showed in \cite{GZ} that, over a suitable ring $R'$ with $R \subsetneqq R' \subsetneqq \Q(a, s)$, the skein module of $S^1 \times S^2$ is freely generated by the class of the empty link $\emptyset$.    As an immediate consequence, it is observed that:
\begin{theorem}[\cite{GZ}]  There exists a unique invariant of framed links $f: \mathcal{L}(S^1\times S^2) \rightarrow \Q(a, s)$ satisfying the HOMFLY-PT skein relations and $f(\emptyset) = 1$.
\end{theorem}
The proof does not give a practical method for computing $f(L)$.  In the present paper, we provide a direct approach which we now describe.

The field of rational functions $\Q(a,s)$ is clearly an $R$-module.  Hence, this $S^1\times S^2$ HOMFLY-PT invariant $f$ extends to an $R$-module homomorphism $S(S^1\times S^2) \rightarrow \Q(a,s)$ which we also denote as $f$.  Following \cite{GZ}, the inclusion $i: S^1 \times D^2 \hookrightarrow S^1 \times S^2$ induces a surjection of skein modules, $i_* : \hfs \rightarrow S( S^1 \times S^2) $.  We let $F : S(S^1\times D^2) \rightarrow \Q(a,s)$ denote the composition $F = f \circ i_*$.  By a slight abuse of terminology, we also refer to $F$ as the {\it $S^1 \times S^2$ HOMFLY-PT invariant}.

  Turaev \cite{Tu} showed that the skein module $\hfs$ has a natural basis, $\{A_\lambda A_{-\mu} \}$, indexed by pairs of partitions, $\lambda$ and $\mu$.  The $A_\lambda A_{-\mu}$ are geometrically simple links (see Section \ref{sec:TuraevB}),  
and it is straightforward to write an arbitrary link $L \subset S^1\times D^2$ as a linear combination of the $A_\lambda A_{-\mu}$ by applying the HOMFLY-PT skein relations. Therefore, for computation it suffices to evaluate $F(A_\lambda A_{-\mu})$.  

The value of $F$ on the links $A_\lambda A_{-\mu}$ turns out to be related to Legendrian knot theory.  Each of the links $A_\lambda A_{-\mu}$ may be viewed as a Legendrian link with respect to the standard contact structure on $S^1 \times D^2$.  In this setting, there are invariants of Legendrian links known as ruling polynomials defined by Chekanov and Pushkar in \cite{ChP}.  The $2$-graded ruling polynomials of the basis links $A_\lambda A_{-\mu}$ are computed in detail in \cite{R2}, and the results are recalled in Theorem \ref{thm:R2Basic} below.

\begin{theorem} \label{the:HOMFLY} Let $F: \hfs \rightarrow \Q(a,s)$ denote 
the $S^1\times S^2$ HOMFLY-PT invariant.  For any partitions $\lambda$ and $\mu$ we have 
\[
F(A_\lambda A_{-\mu}) = R^2_{A_\lambda A_{-\mu}}(z)
\]
 where $R^2$ denotes the $2$-graded ruling polynomial.
\end{theorem} 

Such an explicit description allows us to deduce in Corollary \ref{cor:poly} below that for any $L \subset S^1 \times S^2$ the HOMFLY-PT rational function, $f(L)$, is in fact a Laurent polynomial in $\Z[a^{\pm 1}, z^{\pm 1}]$.  

We give two proofs of Theorem \ref{the:HOMFLY}.  Our primary approach makes use of results from Legendrian knot theory which relate Legendrian link invariants such as the Thurston-Bennequin number and ruling polynomials with the HOMFLY-PT skein module.  In turn, Theorem \ref{the:HOMFLY} then allows us to extend a HOMFLY-PT polynomial estimate for the Thurston-Bennequin number to the $S^1\times S^2$ setting, see Corollary \ref{cor:TBEst}.  Note that similar estimates have been shown to hold in certain contact lens spaces in the work of Cornwell \cite{C1} \cite{C2}.  Whether close connections between Legendrian knots and HOMFLY-PT type invariants can be established in more general classes of contact manifolds remains an interesting direction for further investigation.

The remainder of this article is organized as follows.  In Section 2, a number of known results concerning skein modules are recalled for later use.    Section 3 contains background in Legendrian knot theory including a discussion of Legendrian knots in $S^1\times S^2$.  Section 4 contains the proof of Theorem \ref{the:HOMFLY} and its corollaries.  In addition, we provide an example to illustrate the use of Theorem \ref{the:HOMFLY} in computing the $S^1\times S^2$ HOMFLY-PT polynomial.   Finally, we conclude in Section 5 with an alternate proof of Theorem \ref{the:HOMFLY} which is less reliant on Legendrian knot theory.  Instead, change of basis results of Koike \cite{K} concerning symmetric functions are translated to the skein module setting where they are used in combination with a result from \cite{R2} to evaluate $F(A_\lambda A_{-\mu})$.

\subsection{Acknowledgements}  This work was initiated through the PRUV program at Duke University.  We thank David Kraines for supervising the program and encouraging our participation.  Also, we thank Lenny Ng and Yo'av Rieck for their interest in the project, and Hugh Morton for e-mail correspondence related to this work.  
M. Lavrov received support from NSF CAREER grant DMS-0846346.

\section{Background on skein modules}

When working with the skein module $\hfs$ we view the interior of $S^1 \times D^2$ topologically as $S^1 \times \R^2$.  
We use the projection $S^1 \times \R^2\rightarrow S^1 \times \R$, $(x,y,z) \mapsto (x,y)$ to present framed links via link diagrams in an annulus.  The framing is assumed to be the blackboard framing.   Moreover, we use the homeomorphism $S^1 \cong [0,1]/\{0,1\}$ to view our projection annulus as the planar region $[0,1]\times \R$ with the lines $x=0$ and $x=1$ identified.  

With this choice of projection annulus fixed, a product structure arises on $\hfs$ from stacking diagrams.  That is, given annular diagrams $K$ and $L$ we define their product $K \cdot L$ by stacking $K$ vertically above $L$:
\[
K \cdot L = \raisebox{-7.2ex}{\includegraphics[height=16.2ex]{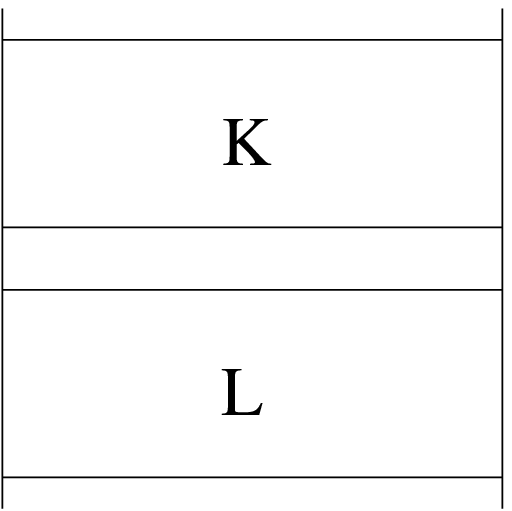}}.
\]  
The resulting product on $\hfs$ is  commutative, and the empty diagram is a multiplicative identity element.

\subsection{Links in $S^1 \times S^2$}  Any link in $S^1 \times S^2$ may be isotoped into $S^1 \times D^2$.  In this way, the set of isotopy classes of framed links,  $\mathcal{L}(S^1 \times S^2)$, may be viewed as $\mathcal{L}(S^1 \times D^2)/\sim$ where $\sim$ denotes the equivalence relation corresponding to isotopy of links in the larger space.   For instance,  link diagrams related by the Move A pictured in Figure \ref{fig:ExtraMove} represent isotopic framed links in $S^1\times S^2$.

To see this, note that $S^1 \times S^2$ can be obtained from $S^1\times D^2$ by attaching a $2$-handle along a meridian curve, $\mathit{pt} \times S^1$, and then attaching a $3$-handle to the remaining boundary.  
Move A corresponds to isotoping the pictured arc of the link across the new $2$-handle.

\begin{figure}
\centerline{ $\raisebox{-7ex}{\includegraphics[scale=.5]{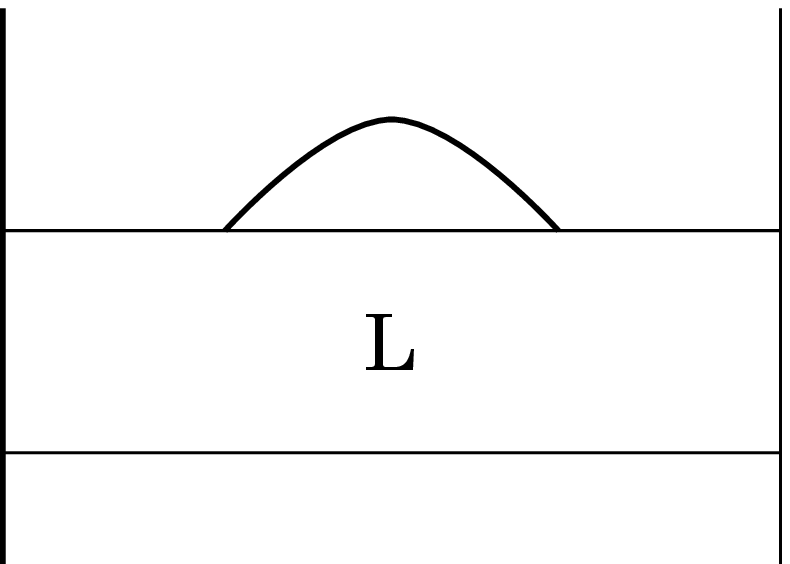}} \quad \leftrightarrow \quad \raisebox{-7ex}{\includegraphics[scale=.5]{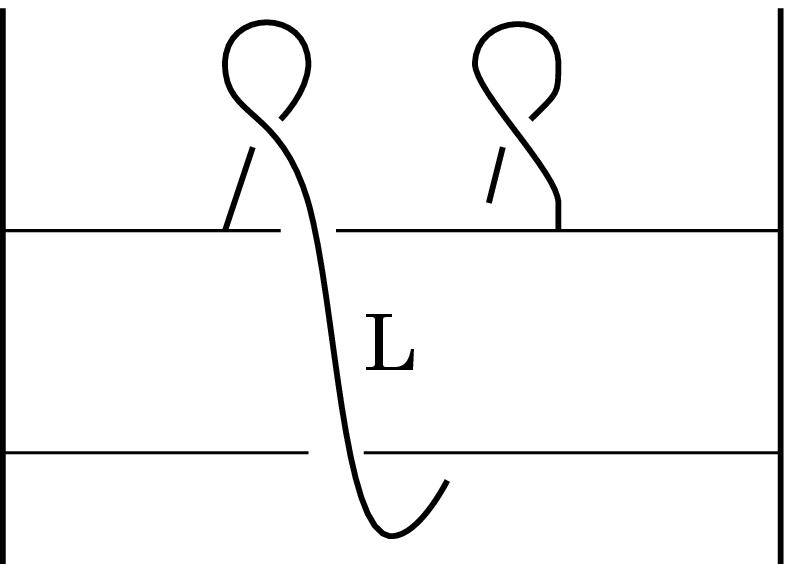}}$}
\caption{Move A}
\label{fig:ExtraMove}
\end{figure}

The two half twists that appear in the diagram on the right in Move A arise as follows.  We can view a framed link as a union of embedded annuli.  Since our diagrams represent links with the blackboard framing, these annuli can be viewed in the projection as thickened versions of the knot diagram.  
During Move A, the portion of the annulus pictured on the left of Figure \ref{fig:ExtraMove} is isotoped across the $2$-handle so that it remains tangent to the core of the $2$-handle during the isotopy.  A collar neighborhood of the boundary of the core of the $2$-handle  is simply a meridian curve, $m$, with blackboard framing where  we isotope $m$ slightly so that in the projection the left (resp. right) half of $m$ runs in front of (resp. behind) $L$.  There are two crossings between the pictured arc of $L$ and $m$.  After the isotopy, they become the self crossings of $L$ on the right of Figure \ref{fig:ExtraMove}.

\begin{remark}  Although we will not need to use it explicitly, it is standard that two link diagrams in $S^1\times \R$ represent isotopic (framed) links in $S^1 \times S^2$ if and only if they can be related by Move A in combination with the usual (framed) Reidemeister moves.  A more general result describing the effect of surgery on isotopy classes of links appears as part of the Kirby calculus.  Specifically, assume $N$ and $M$ are $3$-manifolds with $N$ obtained from attaching a $2$-handle along $\partial M$.  Then, two framed links in $M$ are isotopic in $N$ if and only if they are related by isotopies in $M$ and band connected sums with the attaching curve of the $2$-handle with framing orthogonal to $\partial M$.  
\end{remark}

\subsection{Bases for the skein module $S(S^1\times D^2)$}  \label{sec:TuraevB} For each $m \geq 1$, let $A_m$ denote the knot diagram that winds $m$ times around the annulus with $m-1$ positive crossings.  This diagram may be viewed as the closure of the braid $\sigma_1 \sigma_2 \cdots \sigma_{m-1}$ where strands are numbered from top to bottom.  The diagrams $A_m$ are oriented in the direction of the positive $x$-axis, and we let $A_{-m}$ denote $A_m$ with its orientation reversed.

Recall that a partition $\lambda = (\lambda_1,\ldots, \lambda_\ell)$ is a finite (possibly empty) sequence of positive integers which is non-increasing.  If $\sum_i \lambda_i = n$, we say $\lambda$ is a partition of $n$ and write $\lambda \vdash n$.  Given partitions $\lambda = (\lambda_1,\ldots, \lambda_\ell)$ and $\mu = (\mu_1, \ldots, \mu_k)$, we let $A_\lambda = A_{\lambda_1} \cdots A_{\lambda_\ell}$ and $A_{-\mu} = A_{-\mu_1} \cdots A_{-\mu_k}$, see Figure \ref{fig:BF}.
Note that $A_\emptyset = 1$. 
\begin{figure}
\centerline{ \includegraphics[scale=.5]{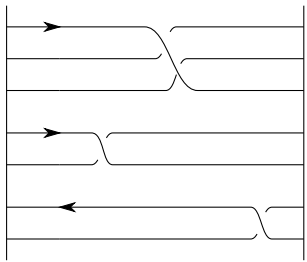}}
\caption{The link $A_\lambda A_{-\mu}$ where $\lambda = (3,2)$ and $\mu = (2)$.}
\label{fig:BF}
\end{figure}

Turaev proved in \cite{Tu} that  
the skein module $\hfs$ is a free $R$-module with linear basis 
\[
\{ A_\lambda A_{-\mu} \, | \, \lambda \vdash n_1, \mu \vdash n_2 ; n_1, n_2 \geq 0
\}.
\]
In fact this statement remains true over $\Z[a^{\pm 1}, z^{\pm 1}]$.  We will refer to this collection as {\it Turaev's geometric basis}.  

We let $\C^+$ (resp. $\C^-$) denote the subalgebras generated by $A_m$ with $m > 0$ (resp. $m < 0$) and $1$.  It follows from the form of Turaev's basis that multiplication gives an isomorphism
\begin{equation} \label{eq:AnIso}
\C^+ \otimes \C^- \cong \hfs.
\end{equation} 
Given an oriented link diagram $L \subset S^1 \times \R$, we let $L^*$ denote the diagram obtained by reversing the orientation of all components of $L$.  This operation induces an algebra automorphism $\psi : \hfs \rightarrow \hfs$ with $\psi(A_m) = A_{-m}$, $m \geq 1$.  In particular, $\psi$ restricts to an isomorphism between $\C^+$ and $\C^-$.

\subsubsection{Meridian eigenvector basis}

The most readily available description of the $S^1\times S^2$ invariant, $F$, (see Theorem \ref{the:F1} below) involves a more complicated basis for $\hfs$ which is denoted $\{Q_{\lambda, \mu}\}$ and referred to as the {\it meridian eigenvector basis}.  This basis was shown to exist in \cite{MH}, and we now recall some details of the explicit construction of the basis elements $Q_{\lambda, \mu}$ from \cite{HM}.

We begin by introducing skein elements $h_n \in S(S^1\times D^2)$ for $n \geq 0$.  When $n=0$, $h_0 = 1$.  For $n >0$, recall that to a permutation $\pi \in S_n$ one can associate a positive  permutation braid $\omega_\pi$.  Such a braid is characterized up to isotopy, by the requirement that strands $i$ and $j$ with $i < j$ have a single positive crossing (the upper strand crosses over the lower strand) if and only if $\pi(i) > \pi(j)$, and otherwise do not meet.  See Figure \ref{fig:Permutation}.    Now, $h_n$ is defined as the linear combination
\[
\frac{1}{\alpha_n} \sum_{\pi \in S_n} s^{l(\pi)} \omega_\pi
\]
where $l(\pi)$ is the number of crossings of $\omega_\pi$ and 
\[
\alpha_n = s^{n(n-1)/2} [n] \, [n-1] \cdots [1] 
\]
with $[i] = \frac{s^i-s^{-i}}{s-s^{-1}}$.  We use  the notation $h_n^* = \psi(h_n)$  for the orientation reverse of $h_n$.

\begin{figure}
\centerline{ \includegraphics[scale=.5]{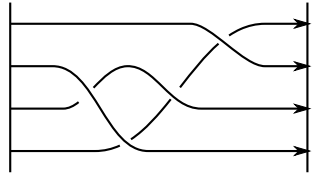}}
\caption{The positive permutation braid $\omega_\pi$ where $\pi$ is the $3$-cycle $(1\, 2\, 4)$.}
\label{fig:Permutation}
\end{figure}

  Supposing $\lambda = (\lambda_1, \ldots, \lambda_\ell)$ and $\mu = (\mu_1, \ldots, \mu_{\ell'})$, Hadji and Morton define
$Q_{\lambda, \mu}$ as the determinant of the $(\ell + \ell') \times (\ell + \ell')$ matrix with diagonal entries $h^*_{\mu_{\ell'}}, \ldots, h^*_{\mu_1}, h_{\lambda_1}, \ldots, h_{\lambda_\ell}$.  The entries in the first $\ell'$ rows consist of elements $h^*_n$ with subscript decreasing from left to right where if $n <0$ we take $h_n=0$.  The entries in the remaining $\ell$ rows are of the form $h_n$ with the subscript now increasing from left to right.

For example, if $\lambda = (4,1)$ and $\mu = (3,2)$, then
\[
Q_{\lambda,\mu} = \det \left[\begin{array} {cccc} h_2^* & h_1^* & 1 & 0 \\
																									h_4^* & h_3^* & h_2^* & h_1^* \\
																									h_2 & h_3 & h_4 & h_5 \\
																									0 & 0 & 1 & h_1
																									\end{array}\right]
\]
  
\begin{remark} (i)  By convention, $Q_{\emptyset, \emptyset}$ is the unit element $1$.

\noindent (ii)
As noted in \cite{HM}, the orientation reverse map $\psi$ satisfies 
\begin{equation} \label{eq:reverse}
\psi(Q_{\lambda, \mu}) = Q_{\mu, \lambda}.
\end{equation}
\end{remark}

We now consider a linear map $\varphi: \hfs \rightarrow \hfs$  arising from inserting an additional loop around link diagrams:
\[
\varphi(X) = \raisebox{-5.4ex}{\includegraphics[height=14.4ex]{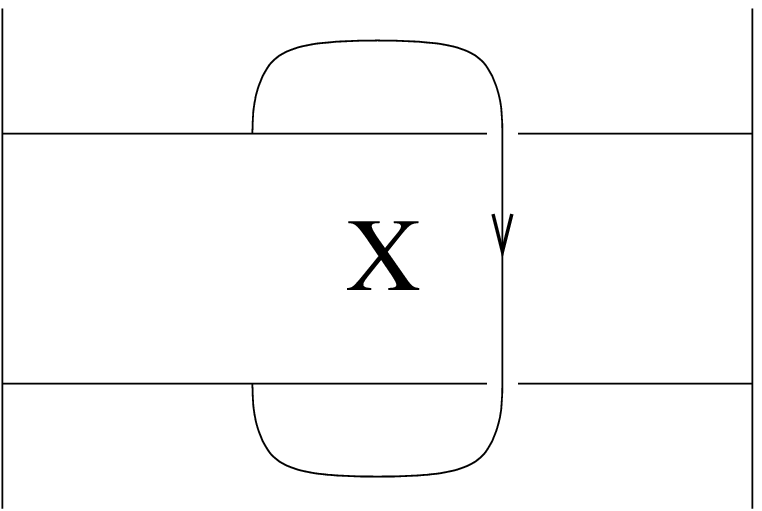}} \,\, .
\]
Since the new component is required to be isotopic to the meridian of $S^1\times D^2$, $\varphi$ is referred to as the {\it meridian map}. 
There is a variant on the meridian map, $\bar{\varphi}: \hfs \rightarrow \hfs$, where the orientation of the new component is reversed.   

As the terminology suggests, the $Q_{\lambda, \mu}$ are eigenvectors of the meridian maps.  The eigenvalues may be described in terms of the Young diagrams of $\lambda$ and $\mu$.  If $\lambda = (\lambda_1, \ldots, \lambda_\ell)$, then the Young diagram of $\lambda$ which we denote as, $Y(\lambda)$, is an arrangement of boxes in $\ell$ rows with $\lambda_i$ boxes in the $i$-th row.  Each row of boxes is left justified.  For example, the partition $\lambda = (5, 2)$ has Young diagram $Y(\lambda) = \raisebox{-1ex}{\includegraphics[scale=.5]{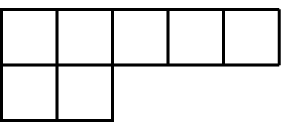}}$.

A box $x \in Y(\lambda)$ has {\it content} $\mathit{ct}(x) = j-i$ if $x$ sits in the $i$-th row of and $j$-th column of $\lambda$.  (As with matrices, rows are enumerated from top to bottom and columns are enumerated from left to right.)  We define the content polynomial of a partition $\lambda$ by
\[
C_\lambda(s) = \sum_{x \in Y(\lambda)} s^{\mathit{ct}(x)}.
\]
Continuing the example above, the contents of the boxes of $Y(\lambda)$ for the partition $\lambda = (5,2)$ are given by
\settoheight{\unitlength}{\boxed{1}}
\begin{center}
\raisebox{0ex}{
\begin{picture}(5,2)
\put(0,0){\line(0,1){2}}
\put(1,0){\line(0,1){2}}
\put(2,0){\line(0,1){2}}
\put(3,1){\line(0,1){1}}
\put(4,1){\line(0,1){1}}
\put(5,1){\line(0,1){1}}
\put(0,0){\line(1,0){2}}
\put(0,1){\line(1,0){5}}
\put(0,2){\line(1,0){5}}

\put(0,0){\makebox(1,1){-1}}
\put(0,1){\makebox(1,1){0}}
\put(1,0){\makebox(1,1){0}}
\put(1,1){\makebox(1,1){1}}
\put(2,1){\makebox(1,1){2}}
\put(3,1){\makebox(1,1){3}}
\put(4,1){\makebox(1,1){4}}
\end{picture}}.
\end{center}
The content polynomial of $\lambda$ is therefore $C_\lambda(s) = s^4 + s^3 + s^2 + s + 2 + s^{-1}$.  Note that the content polynomial $C_\lambda$ uniquely determines the partition $\lambda$ since $Y(\lambda)$ may be reconstructed from the coefficients.

\begin{theorem}[\cite{HM}] \label{lem:HFEig} The skein elements $Q_{\lambda, \mu}$ form an $R$-basis for $\hfs$ and are eigenvectors for the meridian maps.  The corresponding eigenvalues are given by
\[
c_{\lambda, \mu} = z \,(a\, C_\lambda(s^2)-a^{-1}\, C_\mu(s^{-2})) + \frac{a-a^{-1}}{z}
\]
and satisfy
\[
\varphi(Q_{\lambda, \mu}) = c_{\lambda, \mu} Q_{\lambda, \mu} \quad \mbox{and} \quad \bar{\varphi}(Q_{\lambda, \mu}) = c_{\mu, \lambda} Q_{\lambda, \mu}.
\]
\end{theorem}

If either $\lambda$ or $\mu$ is empty, we introduce the shortened notation 
\[
Q_\lambda =Q_{\lambda, \emptyset},  \quad  Q^*_\mu = \psi(Q_\mu) = Q_{\emptyset, \mu}.
\] 
The skein elements $\{Q_{\lambda}\}$ form a basis for the subalgebra $C^+$ and were originally defined in \cite{AM} in a different manner.  The equivalence of this alternate definition with the determinant used to define the $Q_\lambda$ in the present article was established in \cite{Lu}.  Moreover, it is shown in \cite{AM} and \cite{Lu} that the $Q_\lambda$ satisfy the same multiplication rule as the Schur symmetric functions.  (From the determinant perspective this is immediate since the formula for $Q_\lambda$ is identical to the well-known Jacobi-Trudy identity from the theory of symmetric functions; see, for instance \cite{St}.)  In particular, the structure constants are positive integers.

Combining these results with the algebra isomorphism (\ref{eq:AnIso}) gives the following:

\begin{theorem}[\cite{AM, Lu}] \label{thm:ProdBasis} The collection $\{Q_\lambda Q^*_\mu\}$ forms a basis for $\hfs$.  Moreover, the product of any two elements of this basis has positive integer coefficients when written as a linear combination of the $Q_\lambda Q^*_{\mu}$.  
\end{theorem}

\section{Legendrian links and ruling polynomials}

Recall that a {\it contact structure} on a $3$-manifold $M$ is a two plane distribution $\xi \subset TM$ which is maximally non-integrable.  A smooth link $L$ in a contact $3$-manifold $(M, \xi)$ is {\it Legendrian} if $L$ is every tangent to $\xi$.  

We consider Legendrian links in $S^1 \times \R^2$ with respect to the standard contact structure given as the kernel of the $1$-form $dz - y \, dx$.  (The $x$-coordinate is circle-valued.)   Typically, a Legendrian link $L \subset S^1\times \R^2$ is presented via its projection to the $xz$-annulus, $\pi_{xz}(L)$, which we refer to as the {\it front diagram} for $L$.  Generically, a front diagram consists of a union of closed curves in the $xz$-annulus satisfying the following conditions:
\begin{itemize}
\item The diagram $D$ does not have any vertical tangencies.
\item All crossings of $D$ are transverse double points.
\item The components of $D$ are immersed except for semi-cubical cusp singularities. 
\end{itemize}
Moreover, any such diagram lifts to a Legendrian link in $S^1 \times \R^2$ by taking the $y$-coordinate to be the slope, $\displaystyle \frac{dz}{dx}$.  At crossings of a front diagram, it is typical not to indicate which strand appears on top since the one with lesser slope is necessarily the overstrand.  (The $y$-axis is oriented into the page.) An example of a front diagram appears in Figure \ref{fig:RulEx}.

A {\it Legendrian isotopy} is a smooth isotopy of a Legendrian link through other Legendrian links.   Two front diagrams represent Legendrian isotopic links if and only if they are related by a sequence of the so-called Legendrian Reidemeister moves, see for example \cite{R2}.  (One can also pass crossings and cusps across the line $x=0$ which is identified with $x=1$.)  

A Legendrian link $L \subset \st$ has a natural framing, the contact framing, which agrees with the blackboard framing of the projection of $L$ to the $xy$-annulus, $\pi_{xy}(L)$.  Legendrian isotopic links are isotopic as framed links, and we will always use the contact framing when viewing $L$  as an element of a skein module.  The contact framing does not agree with the blackboard framing of the front diagram, $\pi_{xz}(L)$.  Rather it corresponds to the blackboard framing of a diagram obtained from $\pi_{xz}(L)$ by adding a loop with a negative crossing at each right cusp.   (This can be seen, for instance, by using the resolution procedure from \cite{NgTr} to obtain the projection of a Legendrian isotopic link to the $xy$-annulus.)  Thus if we smooth the cusps of $\pi_{xz}(L)$ to view the front diagram as an ordinary link diagram, then we have 
\begin{equation}
[L] = a^{-c(L)} [\pi_{xz}(L)]
\end{equation}
in $\hfs$ where $c(L)$ denotes $\frac{1}{2}$ the number of cusps of $L$ or, equivalently, the number of right cusps of $L$.   

\subsection{Normal rulings}  We now recall a certain combinatorial structure on a front diagram of a Legendrian link $L$ called a normal ruling.  Normal rulings were introduced independently in \cite{ChP} and \cite{F}.   

\begin{figure}
\centerline{ \includegraphics[scale=.6]{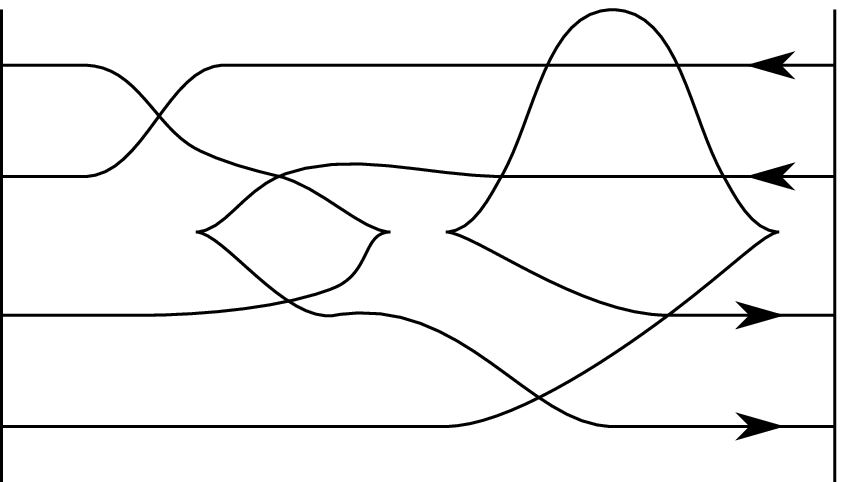} \quad \includegraphics[scale=.6]{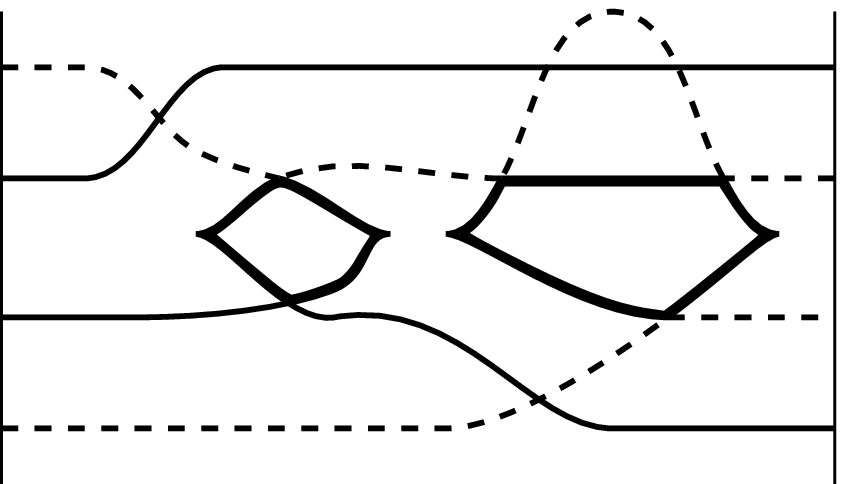} }
\caption{A front diagram for a Legendrian link $L$, and a $2$-graded normal ruling of $L$.}
\label{fig:RulEx}
\end{figure}

Let $L \subset \st$ be a Legendrian link such that the crossings and cusps of the front diagram of $L$ all have distinct $x$-coordinates.  
A {\it normal ruling}, $\rho$, of $L$ is a certain type of division of the points of its front diagram into pairs, and an example appears in Figure \ref{fig:RulEx}.   In detail, at all values of $x$ not containing crossings or cusps, the points of the front diagram are divided into pairs by a fixed point free  involution $\rho_x$ that depends continuously on $x$.   Strands that meet at a cusp should be paired near the cusp point, and the pairing of strands not meeting at the cusp should be the same before and after the cusp.  

Strands that meet at a crossing are not allowed to be paired near the crossing.  In addition, the involutions are required to satisfy the following continuity property at crossings.  For values of $x$ in an open interval containing the $x$-value of a crossing point, we can find continuous paths with monotonically increasing $x$-coordinates that cover the corresponding portion of the front diagram in a one-to-one manner except at the crossing point which is contained in exactly two paths.  The involutions $\rho_x$, should then divide these paths into pairs.  At the crossing, either the two paths that meet cross one another, or they both turn a corner.  In the latter case we refer to the crossing as a {\it switch}.  

Finally, there is a restriction near switches known as the {\it normality condition}.  At a switch, the two paths that meet are not allowed to be paired with each other, and hence they are each paired with a {\it companion strand} above or below in the front diagram.  If we consider intervals on the $z$-axis arising from connecting the switch strands to their companion strands, then near the switch either the two intervals should be disjoint or one should be contained in the other.  That is, out of the six possibilities for the vertical ordering of the switching strands and their companion strands near the switch, only the three pictured in Figure \ref{fig:NormC} are allowed.    

\begin{figure}
\centerline{ \includegraphics[scale=.8]{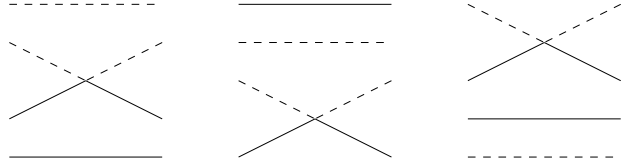}}
\caption{The normality condition.}
\label{fig:NormC}
\end{figure}

If, in addition, the involutions $\rho_x$ always reverse the orientation of $L$, then we say $\rho$ is {\it $2$-graded}.  

The {\it $2$-graded ruling polynomial} of $L$ is defined by
\[
R^2_L(z) = \sum_\rho z^{j(\rho)}
\]
where the sum is over all $2$-graded normal rulings and $j(\rho) = \#\mbox{switches} - \#\mbox{right cusps}$.  The $2$-graded ruling polynomial is a Legendrian isotopy invariant, see \cite{ChP}.

\begin{remark}  In general, $p$-graded ruling polynomials may be defined where $p$ is any divisor of twice the rotation number of $L$.   For this, and for a more formal presentation of the definition of normal ruling, see any of  \cite{ChP}, \cite{R2}, or \cite{LR}.
\end{remark}

\subsection{Normal rulings and the HOMFLY-PT skein module} \label{sec:NR}

Next, we review results relating Legendrian links with the HOMFLY-PT skein module.  Given a Legendrian link $L \subset \st$, the expansion of $[L]$ in $\hfs$ with respect to Turaev's geometric basis may be viewed as a polynomial in $a^{\pm 1}, z^{\pm 1}$ and the infinite collection of variables $A_m$, $m \in \Z \setminus \{0\}$.  We refer to this polynomial as the {\it framed HOMFLY-PT polynomial} of $L$ and denote it by, $H_L \in \Z[a^{\pm 1}, z^{\pm 1}, A_{\pm1}, A_{\pm2}, \ldots]$.  The (unframed) {\it HOMFLY-PT polynomial}, $P_L$, is then defined as $P_L = a^{-w(L)} H_L$  where $w(L)$ denotes the writhe of $L$ which is a signed sum of crossings of a diagram for $L$ (with blackboard framing) in the $xy$-annulus;  $P_L$ is independent of the framing of $L$.

There are two integer invariants of Legendrian links in $J^1(S^1)$, the {\it Thurston-Bennequin number} and the {\it rotation number}, which we denote as $\mathit{tb}(L)$ and $r(L)$ respectively.  In this context, we define $\mathit{tb}(L)$ as  the writhe of a front diagram of $L$ minus the number of right cusps, or equivalently as the writhe of the projection to the $xy$-annulus.  The rotation number of a $1$-component link is defined as $r(L) = \frac{1}{2}[ d(L) - u(L)]$ where $d(L)$ (resp. $u(L)$) denotes the number of downward (resp. upward) oriented cusps.  For a multi-component link we take $r(L)$ to be the {\it sum} of the rotation numbers of the components of $L$.  

The following inequality is proved in \cite{CG} (see also Theorem 6.1 from \cite{R2} for a proof matching our conventions).

\begin{theorem}[\cite{CG}]  \label{thm:CG} For any Legendrian link $L \subset \st$, $\mathit{tb}(L) + |r(L)| \leq - \deg_a P_L$ where $\deg_a P_L$ denotes the degree of $P_L$ if viewed as a Laurent polynomial in $a$. 
\end{theorem}    
As a consequence, the unframed HOMFLY-PT polynomial,  $H_L$, satisfies $\deg_aH_L \leq 0$.

Now, consider an $R$-module homomorphism $G: S(S^1\times D^2) \rightarrow R$ defined by the requirement that on Turaev's basis
\[
G(A_\lambda A_{-\mu}) = R^2_{A_\lambda A_{-\mu}}(z).
\]
(We follow the convention that the empty diagram has a unique normal ruling so that $G(1) = 1$.)  Theorem \ref{the:HOMFLY}, which remains to be proven, claims that $G$ is equal to the $S^1\times S^2$ HOMFLY-PT invariant which we have earlier denoted as $F$.

The following properties of $G$ are established\footnote{The statement here is a slight reformulation of \cite{R2}, Theorem 6.3 which is stated in terms of an unframed version of $G$.} in \cite{R2}.  

\begin{theorem}[\cite{R2}] \label{thm:R2} For any Legendrian $L \subset \st$, 
\[
G(L) = R^2_L(z) + p(a^{-1})
\]
where $p(a^{-1})$ denotes a Laurent polynomial in $\Z[a^{\pm 1}, z^{\pm 1}]$ which only contains negative powers of $a$.
\end{theorem}

Moreover, a relationship with the Meridian eigenvector basis is established.  In the following, we again abbreviate $Q_{\lambda} = Q_{\lambda, \emptyset}$ and $Q_\mu^* = \psi(Q_\mu) = Q_{\emptyset, \mu}$.

\begin{theorem}[\cite{R2}] \label{thm:InnerProduct} For any $\lambda$ and $\mu$,
\[
G(Q_\lambda Q^*_{\mu}) = \delta_{\lambda,\mu}.
\]
\end{theorem}

\begin{proof}
In Section 5 of \cite{R2}, the  bilinear form $(\cdot, \cdot): \C^+ \times \C^+ \rightarrow R$ defined by $(Q_\lambda, Q_\mu) = \delta_{\lambda,\mu}$ is studied.  In particular, Theorem 5.6 of \cite{R2} shows that $(A_\lambda, A_\mu) = R^2_{A_\lambda A_{-\mu}}(z)$.  Let $g:\C^+ \otimes \C^+ \rightarrow R$ denote the linear map corresponding to $(\cdot, \cdot)$.  We can thus write $G$ as the composition $\hfs \cong \C^+ \otimes \C^- \stackrel{\mathit{id} \otimes \psi}{\rightarrow} \C^+ \otimes \C^+ \stackrel{g}{\rightarrow} R$, and the formula follows.
\end{proof}

\subsection{Legendrian links in $S^1\times S^2$}

The manifold $S^1 \times S^2$ admits a unique tight contact structure.  Work of Gompf \cite{G} allows Legendrian links in $S^1\times S^2$, as well as the equivalence relation of Legendrian isotopy, to be presented in a diagrammatic manner.  Gompf constructs $S^1 \times S^2$ with its tight contact structure as the result of removing two appropriately chosen open balls from $S^3$, viewed as $\R^3$ with a point at infinity, and gluing their boundary spheres, $S_1$ and $S_2$, via an orientation reversing diffeomorphism.  Some care is taken so that the standard contact structure on $\R^3$ produces a contact structure after the identification.  A Legendrian link $L \subset S^1 \times S^2$ can then be studied via its front projection.  

Gompf shows in \cite{G}, Theorem 2.2 that after a Legendrian isotopy the front projection of $L$ may be put into a standard form where the projection lies between the identified spheres $S_1$ and $S_2$.  Moreover, $L$ can be assumed to intersect the spheres $S_1$ and $S_2$ in a standard manner, so that $L$ may be represented by what Gompf calls a {\it Legendrian link diagram in standard form} which is entirely equivalent to a front projection of a link in $S^1\times \R^2$.  In addition, Gompf shows that Legendrians $L_1$ and $L_2$ with diagrams in standard form are Legendrian isotopic if and only if their diagrams are related by a combination of  Legendrian isotopies in $S^1\times \R^2$ and the Move B (along with its reflection across a vertical axis) which is pictured in Figure \ref{fig:MoveB}.

\begin{figure}
\centerline{ $\raisebox{-6ex}{\includegraphics[scale=.5]{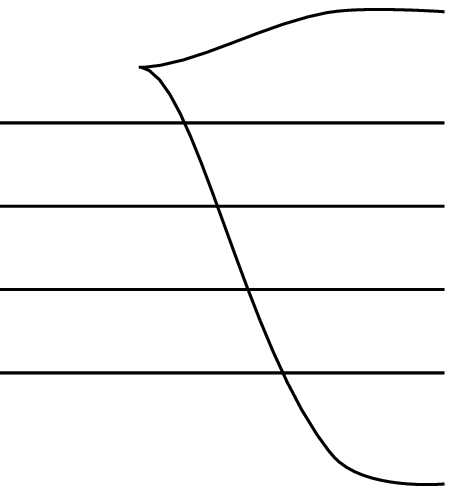}} \quad \leftrightarrow \quad \raisebox{-6ex}{\includegraphics[scale=.5]{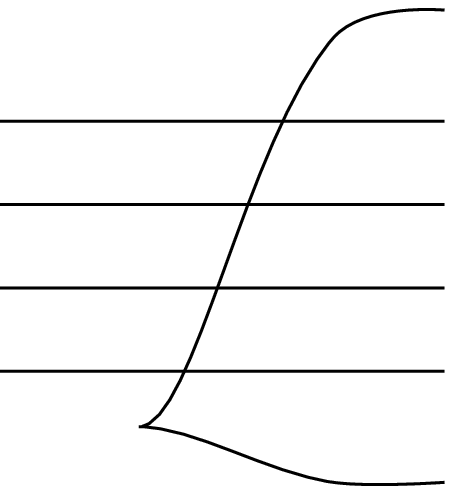}}$}
\caption{Move B}
\label{fig:MoveB}
\end{figure}

Using this diagrammatic perspective, the notion of normal ruling makes sense for a Legendrian link in $S^1 \times S^2$.

\begin{theorem}  \label{thm:Invariance} The $2$-graded ruling polynomial is an invariant of Legendrian links in $S^1 \times S^2$.
\end{theorem}

\begin{proof}  We check invariance of $R^2_L(z)$ under Move B.  Let $L_1$ and $L_2$ denote two diagrams related as in Move B.  There is a clear bijection between those normal rulings of $L_1$ and $L_2$ that do not have switches at any of the pictured crossings, and this bijection preserves the exponent $j(\rho)$ involved in the definition of the ruling polynomial.  To complete the proof we show that these are the only normal rulings of $L_1$ and $L_2$.

Indeed, suppose that $L_1$ has a normal ruling with a switch occurring among the pictured crossings.  (A symmetric argument will apply to $L_2$.)  The strands coming out of the cusp are paired until the lower strand reaches the first such switch.  There, the normality condition implies that the companion strand of the lower half of the switch must be located below the switch.  Moving to the right of this switch, this new pair of strands will intersect unless the upper strand switches again before they meet.  However, the normality condition shows that the companion strand of the lower half of this next switch must also be located below the switch.  Again, we have two strands that are paired by the ruling and will intersect unless a switch occurs along the upper of the two strands.  There are only finitely many crossings, so continuing in this manner will eventually produce two strands paired by the ruling that intersect at a crossing.  
\end{proof}

\section{The $S^1\times S^2$ HOMFLY-PT polynomial}

In this section we prove Theorem \ref{the:HOMFLY} and provide an example to show how this result allows the computation of the HOMFLY-PT invariant for links in $S^1\times S^2$.  

\subsection{Description of the $S^1\times S^2$ invariant in terms of the meridian eigenvector basis}

The values of the $S^1\times S^2$ invariant, $F$, on the meridian eigenvector basis vectors are quite simple.
 
\begin{theorem} \label{the:F1} Let $F: \hfs \rightarrow \Q(a,s)$ be the $S^1\times S^2$ HOMFLY-PT invariant.  
Then, $F(Q_{\emptyset, \emptyset}) = 1$ and $F(Q_{\lambda, \mu}) = 0$ if either $\lambda$ or $\mu$ is non-empty. 

Moreover, for any framed link $L \subset \st$, we have $F(L) \in R$.
\end{theorem}

\begin{proof}  Recall that $F$ is the composition $\hfs \stackrel{i_*}\rightarrow \hfsone \stackrel{f}{\rightarrow} \Q(a,s)$ where $i_*$ is induced by inclusion.

An application of Move A shows that for any $L \subset \st$ the framed links $\varphi(L)$ and $L \sqcup \fig{HSR6.eps}$ are isotopic in $S^1 \times S^2$.  
At the level of skein modules this gives us the equality
\[
i_*(\varphi(X)) = c_{\emptyset, \emptyset} i_*(X)
\]
for any $X \in \hfs$.  Thus, we can compute
\[
c_{\lambda, \mu} F(Q_{\lambda,\mu}) = F( c_{\lambda, \mu} Q_{\lambda, \mu} ) = f( i_*( \varphi(Q_{\lambda, \mu})) ) =  f( c_{\emptyset, \emptyset}  i_*(Q_{\lambda,\mu})) = c_{\emptyset, \emptyset} F(Q_{\lambda, \mu}).
\]
From Theorem \ref{lem:HFEig} we deduce that $F(Q_{\lambda, \mu}) = 0$ unless $\lambda = \mu = \emptyset$.  That $F(Q_{\emptyset, \emptyset})=1$ follows from the definition of $f$.

The second claim holds since by Theorem \ref{lem:HFEig} any framed link is an $R$-linear combination of the $Q_{\lambda, \mu}$.
\end{proof}

\subsection{Legendrian meridian map}  Given a Legendrian link $L \subset S^1 \times D^2$ we let $\nu(L)$  denote the Legendrian link obtained from adding a Legendrian unknot to the front diagram of $L$ as pictured in Figure \ref{fig:LMerid}.  We use $\bar{\nu}(L)$ for the same construction with the opposite orientation taken on the new unknotted component.  It is clear that in the HOMFLY-PT skein module $\hfs$ we have 
\begin{equation} \label{eq:LMerid}
\nu(L) = a^{-1} \varphi(L) \quad \mbox{and} \quad \bar{\nu}(L) = a^{-1} \bar{\varphi}(L).
\end{equation}

\begin{figure}
\centerline{ \raisebox{-6ex}{\includegraphics[scale=.5]{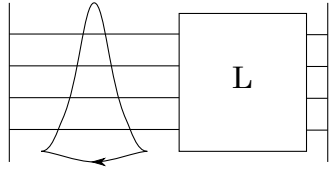}} }
\caption{The Legendrian meridian map applied to a Legendrian link $L$.}
\label{fig:LMerid}
\end{figure}

\begin{lemma} \label{lem:LMerid} For any Legendrian $L$, we have $R^2_{\nu(L)}(z) = R^2_{\bar{\nu}(L)}(z) = z^{-1} R^2_L(z)$.
\end{lemma}
\begin{proof}
This follows from Theorem \ref{thm:Invariance} since after an application of Move B the unknotted component may be isotoped to be disjoint from the rest of the diagram.  
\end{proof}

\subsection{Proof of Theorem \ref{the:HOMFLY}}

To establish Theorem \ref{the:HOMFLY} we need to show that the homomorphisms $F,G \in \mathit{Hom}_R(\hfs, R)$ agree.  Here, $F$ denotes the $S^1 \times S^2$ HOMFLY-PT invariant and $G$ is the homomorphism discussed in Section \ref{sec:NR} which was defined by $G(A_\lambda A_{-\mu}) = R^2_{A_\lambda A_{-\mu}}(z)$.  Our proof is based on the following lemma which again makes use of the notation $p(a^{-1})$ to denote a term containing only negative powers of $a$.

\begin{lemma} \label{thm:RF} For any Legendrian link $L \subset \st$, 
\[
G(L) = F(L) + p(a^{-1}).
\]
\end{lemma}

Before proving the lemma we complete the proof of Theorem \ref{the:HOMFLY} by showing that the $p(a^{-1})$ term must be zero.  Recall that the coefficient ring $R$ has the form $R_0[a^{\pm 1}]$ where $R_0$ is a subring of $\Q(s)$.  Clearly, Lemma \ref{thm:RF} extends to show that if $y \in \hfs$ is an $R_0$-linear combination of Legendrian links then 
\begin{equation} \label{eq:GFp}
G(y) = F(y) + p(a^{-1}).
\end{equation}  
  The positive permutation braids $\omega_\pi$ used in the definition of $h_n$ and $h_n^*$ are Legendrian links.  Indeed, since at all crossings the strand with lesser slope appears on top, they are represented by front diagrams without cusps.  Since the product of Legendrian links is again Legendrian, it follows that the meridian eigenvector basis vectors $Q_{\lambda, \mu}$ are $R_0$-linear combinations of Legendrian links.  Then, equation (\ref{eq:GFp}) gives
\begin{equation} \label{eq:Almost}
G(Q_{\lambda, \mu}) = \left\{ \begin{array}{cr} 1 + p(a^{-1}) & \mbox{if $\lambda = \mu = \emptyset$, and}  \\
																								0 + p(a^{-1}) & \mbox{else.} \end{array} \right.
\end{equation}

On the other hand, since $h_n = Q_{(n)}$ and $h_n^* = Q^*_{(n)}$ it follows from the definition of the $Q_{\lambda, \mu}$ and from Theorem \ref{thm:ProdBasis} that each $Q_{\lambda, \mu}$ is an integer linear combination of the basis elements $\{ Q_\lambda Q^*_{\mu}\}$ from Theorem \ref{thm:ProdBasis}.  We can then apply Theorem \ref{thm:InnerProduct} to conclude that $G(Q_{\lambda, \mu})$ is an integer.  Thus, the $p(a^{-1})$ terms in (\ref{eq:Almost}) must vanish, and we have established that $F = G$ which completes the proof.

\begin{proof}[Proof of Lemma \ref{thm:RF}.]  Let $L \subset S^1\times D^2$ be a Legendrian link.  In $\hfs$ we can write $L$ as a linear combination
\[
L = \sum_{\lambda, \mu} f_{\lambda,\mu} Q_{\lambda, \mu}
\]
with coefficients $f_{\lambda,\mu} \in R$.  

Using Theorem \ref{thm:R2} along with Lemma \ref{lem:LMerid}, for $k,l \geq 0$ we can compute
\[
\begin{aligned} G(\nu^k\circ\bar{\nu}^l(L)) & = & R^2_{\nu^k\circ\bar{\nu}^l(L)}(z) + p(a^{-1}) \\
																& = & z^{-k-l} R^2_L(z) + p(a^{-1}).
																\end{aligned}\]
On the other hand, using equation (\ref{eq:LMerid}) and Theorem \ref{lem:HFEig} we have
\[
\begin{aligned} G(\nu^k\circ\bar{\nu}^l(L)) & = & G(a^{-k-l} \varphi^k\bar{\varphi}^l(L)) \\
																& = & \sum_{\lambda,\mu} a^{-k-l} G(f_{\lambda,\mu} (c_{\lambda,\mu})^k (c_{\mu, \lambda})^l Q_{\lambda,\mu}) \\
																& = & \sum_{\lambda,\mu} a^{-k-l} (c_{\lambda,\mu})^k (c_{\mu, \lambda})^l \left[f_{\lambda,\mu} G(Q_{\lambda,\mu}) \right].
																\end{aligned}\]

Introducing notation $x_{\lambda,\mu} = f_{\lambda,\mu} G(Q_{\lambda,\mu}) \in R$, we see that for all $k,l \geq 0$
\begin{equation} \label{eq:LinSystem}
\sum_{\lambda,\mu} a^{-k-l} (c_{\lambda,\mu})^k (c_{\mu, \lambda})^l \cdot x_{\lambda, \mu} = z^{-k-l} R^1_L(z) + p(a^{-1}).
\end{equation}

Let $n$ be the largest power of $a$ appearing in any of the $x_{\lambda, \mu}$, and denote the coefficient of $a^n$ in $x_{\lambda, \mu}$ by $w_{\lambda, \mu} \in R_0$. 
Now, Theorem \ref{lem:HFEig}  shows that the multiplier of $x_{\lambda, \mu}$ in (\ref{eq:LinSystem}) has the form  
\[a^{-k-l} (c_{\lambda,\mu})^k (c_{\mu, \lambda})^l = (zC_\lambda(s^2) + z^{-1})^k \, (z C_\mu(s^2) + z^{-1})^l + p(a^{-1}).
\]  
Therefore, equating coefficients of $a^n$ in (\ref{eq:LinSystem}) gives 
\begin{equation} \label{eq:LinSystem2}
\sum_{\lambda,\mu} (zC_\lambda(s^2) + z^{-1})^k \, (z C_\mu(s^2) + z^{-1})^l \cdot w_{\lambda, \mu} = 0
\end{equation}
if $n>0$, and
\begin{equation} \label{eq:LinSystem3}
\sum_{\lambda,\mu} (zC_\lambda(s^2) + z^{-1})^k \, (z C_\mu(s^2) + z^{-1})^l \cdot w_{\lambda, \mu} = z^{-k-l}R^2_L(z)
\end{equation}
if $n =0$.

Let $S$ denote a finite collection of partitions such that $x_{\lambda, \mu} \neq 0$ implies  $\lambda, \mu \in S$.  The $w_{\lambda, \mu}$ satisfy the system of linear equations (\ref{eq:LinSystem2}) or (\ref{eq:LinSystem3}) where we take $0 \leq k, l \leq |S|-1$ so that we have as many equations as unknowns.  Then, the coefficient matrix $A= (a_{(k,l), (\lambda,\mu)})$ has entries $a_{(k,l),(\lambda,\mu)} = (z C_\lambda(s^2) + z^{-1})^k\,(z C_\mu(s^2) + z^{-1})^l $.  Thus, $A$ is the Kronecker product (tensor product) of the matrix $B = (b_{k,\lambda})$  with itself where $b_{k, \lambda} = (z C_\lambda(s^2) + z^{-1})^k$.  The matrix $B$ is a Vandermonde matrix which is non-singular since the content polynomials $C_\lambda(s)$ are uniquely determined by the partition $\lambda$.  Therefore, $A$ is non-singular as well, and the system of linear equations (\ref{eq:LinSystem2}) or (\ref{eq:LinSystem3}) has at most one solution.  

If $n >0$, the solution to (\ref{eq:LinSystem2}) is $w_{\lambda, \mu} =0$ which would contradict the choice of $n$.  Hence, $n \leq 0$, and the constant terms of the $x_{\lambda, \mu}$ must equal the unique solution to (\ref{eq:LinSystem3}) which is $w_{\emptyset,\emptyset} =R^2_L(z)$ and $w_{\lambda, \mu} =0$ when $\lambda$ or $\mu$ is non-empty.  We have thus shown that $x_{\emptyset,\emptyset} = R^2_L(z) + p(a^{-1})$.  This completes the proof since according to Theorem \ref{the:F1}
\[
F(L) = f_{\emptyset,\emptyset} = f_{\emptyset,\emptyset} G(1) = f_{\emptyset,\emptyset} G(Q_{\emptyset,\emptyset}) = x_{\emptyset,\emptyset},
\]
and $G(L)$ also has the form $R^2_L(z) + p(a^{-1})$ by Theorem \ref{thm:R2}.
\end{proof}

\begin{corollary} \label{cor:poly} For any framed link $L \subset S^1\times S^2$, the HOMFLY-PT invariant, $f(L)$, is a Laurent polynomial in $a$ and $z$.
\end{corollary}

\begin{proof}[Proof of Corollary \ref{cor:poly}]  Isotope $L$ into $S^1\times D^2$.  Then, $L$ may be written as a $\Z[a^{\pm 1}, z^{\pm 1}]$-linear combination of the $A_\lambda A_{-\mu}$.  
The result follows since the ruling polynomials $R^2_{A_\lambda A_{-\mu}}(z)$ are Laurent polynomials in $z$ with positive integer coefficients.
\end{proof}

As another corollary of Theorem \ref{the:HOMFLY} we extend the inequality from Theorem \ref{thm:CG}.  For a null-homologous Legendrian link $L \subset S^1 \times S^2$, the notions of Thurston-Bennequin and rotation number have existing meaning.  These invariants may be computed in the same manner as for links in $\st$ using a front diagram for $L$.  However, note that for a link that is nontrivial in homology such a diagrammatically defined Thurston-Bennequin number will not be  invariant under Move B; see \cite{G}, Section 2 for further discussion.

For a null-homologous framed link $L \subset S^1\times \R^2$ we can normalize the $S^1 \times S^2$ HOMFLY-PT invariant of $L$ using the writhe, $w(L)$, to obtain a polynomial, $\widetilde{F}_L = a^{-w(L)} F(L) \in \Z[a^{\pm 1}, z^{\pm 1}]$,  which is an invariant of (unframed) link types in $S^1\times S^2$.  

\begin{corollary} \label{cor:TBEst} For a null-homologous Legendrian link $L \subset S^1\times S^2$, we have the estimate
\[
\mathit{tb}(L) + |r(L)| \leq -\deg_a \widetilde{F}_L.
\]
Moreover, the coefficient of $a^{-\mathit{tb}(L)}$ in $\widetilde{F}_L$ is equal to the $2$-graded ruling polynomial $R^2_L(z)$.
\end{corollary}
\begin{proof}
The first statement follows from Theorem \ref{thm:CG} since Theorem \ref{the:HOMFLY} shows that $\deg_a \widetilde{F}_L \leq \deg_a P_L$.  The second statement is a combination of Theorem \ref{thm:R2} and Theorem \ref{the:HOMFLY}.
\end{proof}

\subsection{Computation}

The $2$-graded ruling polynomials of the basis links $A_\lambda A_{-\mu}$ are computed in \cite{R2}.  In combination with Theorem \ref{the:HOMFLY}, this allows computation of the $S^1\times S^2$ HOMFLY-PT polynomial since it is algorithmic to write a given link as a linear combination of the $A_\lambda A_{-\mu}$.  We now recall the formula for $R^2_{A_\lambda A_{-\mu}}(z)$ (\cite{R2}, Theorem 4.2) which is simplified a bit by Lemma 3.2 of \cite{LR}.

For $m \geq 1$, we let
\[
\langle m \rangle = \sum_{k=0}^{m-1} {m+k \choose 2k+1} z^{2k}, 
\]
and make the convention that $\langle 0 \rangle = z^{-2}$.

\begin{theorem}[\cite{R2}] \label{thm:R2Basic} Let $\lambda$ and  $\mu$ be partitions with $\lambda = (\lambda_1, \ldots, \lambda_\ell)$ and $\mu = (\mu_1, \ldots, \mu_k)$.  Denote by $M_{\lambda, \mu}$ the set of $\ell \times k$ matrices with non-negative integer entries such that entries in the $i$-th row sum to $\lambda_i$ and entries in the $j$-th column sum to $\mu_j$.  Then,
$$
\displaystyle
R^2_{A_\lambda A_{-\mu}}(z) = z^{2 \ell k - \ell - k} \sum_{(b_{ij}) \in M_{\lambda, \mu} } \prod_{i, j} \langle b_{ij} \rangle.
$$
\end{theorem}

\begin{figure}
\[\begin{array}{ccc} \includegraphics[scale=.6]{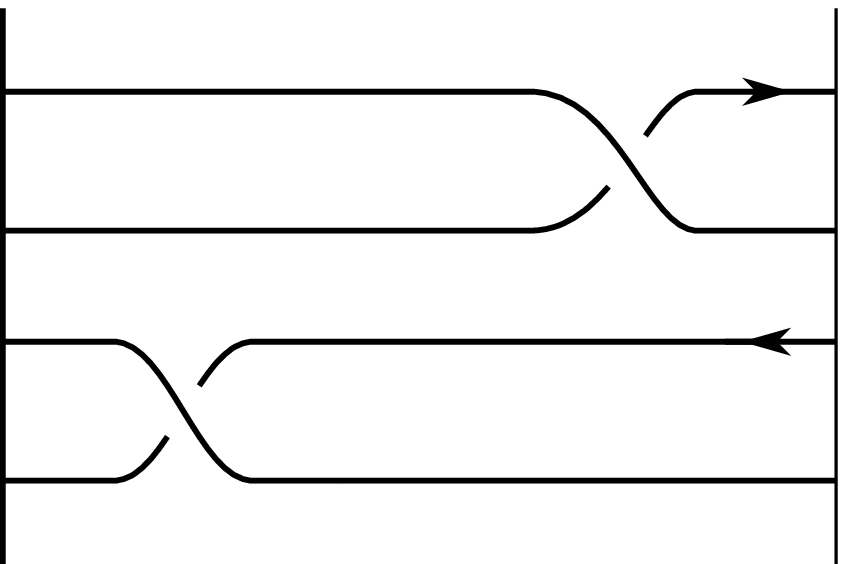} & \quad & \includegraphics[scale=.6]{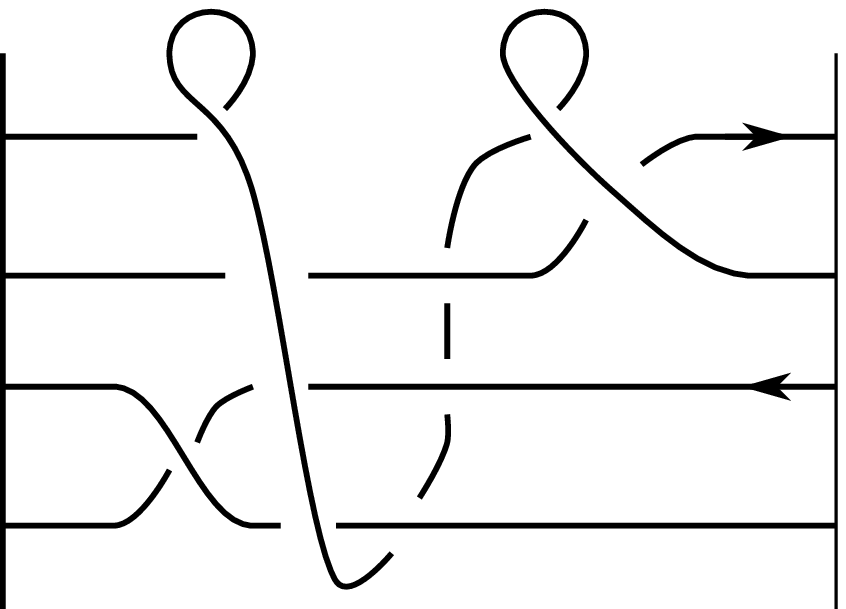} \\
L_1 & \quad & L_2 \end{array} 
\]
\caption{The links $L_1$ and $L_2$.}
\label{fig:Example}
\end{figure}

\begin{example} The links $L_1, L_2 \subset S^1 \times D^2$ pictured in Figure \ref{fig:Example} are isotopic in $S^1\times S^2$ by an application of Move A.   In $\hfs$, $L_1 = A_2 A_{-2}$, and a somewhat laborious computation with skein relations shows that
\[
L_2 = a^2(1+z^2)\, A_2 A_{-2} + a^2 z \, A_{(1,1)} A_{-2} + \big(-a^2[2z^2 +z^4] -  z^2 \big) A_1A_{-1} + (-a^2 +1)(2+2z^2).
\]
Using Theorem \ref{thm:R2Basic}, we have 
\[
R^2_{A_2 A_{-2}}(z) = 2+z^2, \quad R^2_{A_{(1,1)}A_{-2}}(z) = z, \quad R^2_{A_1A_{-1}}(z) =1.
\]
Together with Theorem \ref{the:HOMFLY} this allows us to directly verify the equality
\[
F(L_1)  = 2+z^2 = F(L_2).
\]
\end{example}

\section{The $S^1\times S^2$ HOMFLY-PT polynomial: Alternate approach}

In this section we provide a second proof of Theorem \ref{the:HOMFLY}.  This alternate approach is mostly independent of Legendrian knot theory, including the proof from \cite{R2} of the requisite Theorem \ref{thm:InnerProduct}.  Instead, we appeal to results concerning symmetric functions from the work of Koike \cite{K} to provide explicit transition formulas between the bases $\{Q_{\lambda, \mu}\}$ and $\{Q_\lambda Q^*_\mu \}.$

\subsection{Relation with symmetric functions}

The subalgebra $\C^+ \subset \hfs$ is isomorphic to the ring of symmetric functions $\Lambda({\bf x})$ in an infinite collection of variables ${\bf x} = (x_1, x_2, \ldots)$ with coefficients in $R$.  A standard choice of isomorphism, used for instance in \cite{AM} and \cite{Lu}, is given by identifying the skein elements $h_n$ with the complete homogeneous symmetric functions.  Under this correspondence,  $Q_\lambda$ corresponds to the Schur function, $s_\lambda({\bf x})$, since the determinant formula for $Q_\lambda$ agrees with the Jacobi-Trudy identity.  

For notational clarity, let ${\bf y} = (y_1, y_2, \ldots)$ denote a second collection of variables.  In view of (\ref{eq:AnIso}), we have an algebra isomorphism
\[
\Phi: \hfs \rightarrow \Lambda({\bf x}) \otimes_R \Lambda({\bf y}), \quad Q_\lambda Q_\mu^* \mapsto s_\lambda({\bf x}) \otimes s_\mu({\bf y}).
\]

The relation between the meridian eigenvector basis, $\{Q_{\lambda, \mu}\}$, and the product basis, $\{Q_\lambda Q^*_{\mu}\}$, involves the Littlewood-Richardson coefficients,  $LR_{\lambda, \mu}^\nu$, which may be defined as the structure constants: 
\begin{equation} \label{eq:LR}
s_\lambda({\bf x}) \cdot s_\mu({\bf x}) = \sum_{\nu} LR_{\lambda, \mu}^\nu s_\nu({\bf x}).
\end{equation}

\begin{lemma}[\cite{K}] \label{lem:Koike} In $\hfs$, we have
\[
Q_\lambda Q_\mu^* = \sum_{\tau, \nu, \xi} LR^\lambda_{\nu, \tau} LR^\mu_{\xi, \tau} Q_{\nu, \xi}.
\]
\end{lemma}

\begin{proof}  In \cite{K}, certain elements $[\nu,\xi]_{GL} \in \Lambda_x \otimes \Lambda_y$ are defined for partitions $\nu, \xi$, and the identity
\[
s_\lambda(x) \otimes s_\mu(y)  = \sum_{\tau, \nu, \xi} LR^\lambda_{\nu, \tau} LR^\mu_{\xi, \tau} [\nu, \xi]_{GL}
\]
is established in \cite{K} Theorem 2.3.  (Note that there is typo in the statement of \cite{K} Theorem 2.3.  The original statement of the identity at \cite{K}, page 59,  equation (0.2) is correct.)  

The result follows from applying $\Phi^{-1}$ provided we can show that $\Phi(Q_{\lambda, \mu}) = [\lambda, \mu]_{GL}$.  The elements $[\lambda, \mu]_{GL}$ are defined in Definition 2.1 of \cite{K} using a determinant formula of similar nature to that defining $Q_{\lambda, \mu}$, but with two key differences.  First, the partitions $\lambda$ and $\mu$ are replaced with their conjugates $\lambda'$ and $\mu'$.  Second, the complete homogeneous symmetric functions $h_n$ are replaced with the elementary symmetric functions $e_n$.  

To see that this definition is equivalent, we make use of a well-known algebra involution $\iota: \Lambda \rightarrow \Lambda$ satisfying 
\begin{equation} \label{eq:invo}
\iota(e_n) = h_n \quad \mbox{and} \quad \iota(s_\lambda) = s_{\lambda'}.
\end{equation}
  The first of these equalities is usually taken as the definition of $\iota$; see \cite{St}.  

The tensor product $I = \iota \otimes \iota$ produces an involution on $\Lambda_x \otimes \Lambda_y$, and from (\ref{eq:invo}) and the respective definitions we have
\[
I([\lambda, \mu]_{GL}) = \Phi(Q_{\lambda',\mu'}).
\]
On the other hand, Theorem 2.3 of \cite{K} also contains the identity
\[
[\lambda, \mu]_{GL} = \sum_{\tau, \nu, \xi} (-1)^{|\tau|} LR_{\tau, \nu}^\lambda LR_{\tau', \xi}^\mu s_{\nu}(x) s_{\xi}(y).
\]
Now,  applying $\iota$ to equation (\ref{eq:LR}) shows that $LR_{\lambda, \mu}^{\nu} = LR_{\lambda',\mu'}^{\nu'}$, and making use of this observation we compute 
\[
I([\lambda, \mu]_{GL}) = \sum_{\tau, \nu, \xi} (-1)^{|\tau|} LR_{\tau, \nu}^\lambda LR_{\tau', \xi}^\mu s_{\nu'}(x) s_{\xi'}(y) =
\]
\[
\sum_{\tau, \nu, \xi} (-1)^{|\tau'|} LR_{\tau', \nu'}^{\lambda'} LR_{\tau, \xi'}^{\mu'} s_{\nu'}(x) s_{\xi'}(y) =[\lambda', \mu']_{GL}.
\]
(For the last equality one replaces the summation variables $\tau, \nu,$ and $\xi$ with their conjugates.) 

Combining the previous calculations gives $\Phi(Q_{\lambda', \mu'}) = [\lambda', \mu']_{GL}$ as desired.
\end{proof}

\begin{remark}  The connection between the $Q_{\lambda, \mu}$ and the bases considered by Koike is implicitly pointed out in \cite{HM}.
\end{remark}

\begin{proof}[Alternate proof of Theorem \ref{the:HOMFLY}.]
For any partitions $\lambda$ and $\mu$,
combining Theorem \ref{the:F1} and Lemma \ref{lem:Koike} gives
\[
F(Q_{\lambda}Q^*_{\mu}) = \sum_{\tau} LR^\lambda_{\emptyset, \tau}\, LR^\mu_{\emptyset,\tau} = \sum_{\tau} \delta_{\lambda, \tau}\, \delta_{\mu_,\tau} = \delta_{\lambda, \mu}.
\]
The second equality follows from (\ref{eq:LR}) since $s_\emptyset = 1$.  According to Theorem \ref{thm:InnerProduct}, this shows that $F$ agrees with the homomorphism $G$ from Section \ref{sec:NR}.  This is the desired result. 
\end{proof}

\end{document}